\documentclass{amsart}
\usepackage{amssymb, latexsym}
\usepackage{url}

\begin{document}
 \bibliographystyle{plain}

 \newtheorem{theorem}{Theorem}
 \newtheorem{lemma}{Lemma}
 \newtheorem{corollary}{Corollary}
 \newtheorem{conjecture}{Conjecture}
 \newtheorem{definition}{Definition}
 \newcommand{\mc}{\mathcal}
 \newcommand{\rar}{\rightarrow}
 \newcommand{\Rar}{\Rightarrow}
 \newcommand{\lar}{\leftarrow}
 \newcommand{\lrar}{\leftrightarrow}
 \newcommand{\Lrar}{\Leftrightarrow}
 \newcommand{\zpz}{\mathbb{Z}/p\mathbb{Z}}
 \newcommand{\mbb}{\mathbb}
 \newcommand{\A}{\mc{A}}
 \newcommand{\B}{\mc{B}}
 \newcommand{\cc}{\mc{C}}
 \newcommand{\D}{\mc{D}}
 \newcommand{\E}{\mc{E}}
 \newcommand{\F}{\mc{F}}
 \newcommand{\FN}{\F_n}
 \newcommand{\I}{\mc{I}}
 \newcommand{\J}{\mc{J}}
 \newcommand{\M}{\mc{M}}
 \newcommand{\nn}{\mc{N}}
 \newcommand{\qq}{\mc{Q}}
 \newcommand{\U}{\mc{U}}
 \newcommand{\X}{\mc{X}}
 \newcommand{\Y}{\mc{Y}}
 \newcommand{\C}{\mathbb{C}}
 \newcommand{\R}{\mathbb{R}}
 \newcommand{\N}{\mathbb{N}}
 \newcommand{\Q}{\mathbb{Q}}
 \newcommand{\Z}{\mathbb{Z}}
 \newcommand{\ff}{\mathfrak F}
 \newcommand{\fb}{f_{\beta}}
 \newcommand{\fg}{f_{\gamma}}
 \newcommand{\gb}{g_{\beta}}
 \newcommand{\vphi}{\varphi}
 \newcommand{\whXq}{\widehat{X}_q(0)}
 \newcommand{\Xnn}{g_{n,N}}
 \newcommand{\lf}{\left\lfloor}
 \newcommand{\rf}{\right\rfloor}
 \newcommand{\lQx}{L_Q(x)}
 \newcommand{\lQQ}{\frac{\lQx}{Q}}
 \newcommand{\rQx}{R_Q(x)}
 \newcommand{\rQQ}{\frac{\rQx}{Q}}
 \newcommand{\elQ}{\ell_Q(\alpha )}
 \newcommand{\oa}{\overline{a}}
 \newcommand{\oI}{\overline{I}}
 \newcommand{\dx}{\text{\rm d}x}
 \newcommand{\dy}{\text{\rm d}y}

\title[Martingales and Continued Fractions]{Martingale Differences and the
Metric\\
Theory of Continued Fractions}
\author{Alan~K.~Haynes and Jeffrey~D.~Vaaler}
\address{Department of Mathematics, University of Texas, Austin, Texas
78712 USA}
\email{ahaynes@math.utexas.edu}
\email{vaaler@math.utexas.edu}
 \allowdisplaybreaks


\begin{abstract}
We investigate a collection of orthonormal functions that encodes
information about the continued fraction expansion of real numbers.
When suitably ordered these functions form a complete system of martingale
differences and
are a special case of a class of martingale differences considered
by R.~F.~Gundy.  By applying known results for martingales we
obtain corresponding metric theorems for the continued fraction
expansion of almost all real numbers.
\end{abstract}


\maketitle

\section{Introduction}  Throughout this paper we work with real valued
functions defined on the
compact group $\R/\Z$.  As usual we regard such functions as defined on
$\R$ and
having period $1$.  We frequently regard $\R/\Z$ as a probability space
with respect to a normalized
Haar measure defined on the $\sigma$-algebra of Borel subsets, or
restricted to a
finite sub-$\sigma$-algebra.  We also work with elements of the torsion
subgroup $\Q/\Z$.  If
$\beta$ is a point in $\Q/\Z$ we write $\beta = a/q$ where $q$ is a
positive integer
and $a$ is an integer representing a unique reduced residue class modulo
$q$.  By the {\it height}
of $\beta$ we understand the positive integer $h(\beta) = q$, which is
also the order of $\beta$ in $\Q/\Z$.
In Section 2 we define a countable collection of functions $\fb:\R/\Z\rar
\R$, indexed by points $\beta$
in $\Q/\Z$.  These functions form a complete orthonormal basis for the
Hilbert space $L^2(\R/\Z)$
and also encode information about continued fractions.  The functions
$\fb$, which are the subject
of this paper, were used by Hata \cite{hata1995} in a slightly different
form to obtain interesting identities
for sums over Farey fractions.  In particular, our Theorem
\ref{comporththm} is similar to
\cite[Lemma 3.1]{hata1995}.  We will show that for certain natural
orderings the functions $\fb$, together
with a corresponding sequence of finite $\sigma$-algebras, form a sequence
of martingale differences.  And we
will show that the value of $\fb(\alpha)$ is determined in an elementary
way by the convergents
and intermediate convergents from the continued fraction expansion of
$\alpha$.

Recall that each irrational real number $\alpha$ has an infinite
simple continued fraction expansion
\begin{align*}
\alpha = a_0 + \cfrac{1}{a_1+
            \cfrac{1}{a_2+
             \cfrac{1}{a_3+\dotsb}}}=[a_0; a_1, a_2, a_3, \dots ],
\end{align*}
where $a_0$ is an integer and $a_1, a_2, \dots $ is a sequence of positive
integers uniquely determined by
$\alpha$.  Here we adopt standard notation and terminology as developed in
\cite{khin1964}, \cite{rockett1992},
or \cite{schmidt1980}.  The number $a_n$ is the $n$th partial quotient of
$\alpha$.  If $\alpha$ is rational we write
\begin{equation*}
\alpha = [a_0; a_1, a_2, \dots , a_N]
\end{equation*}
for one of its two finite continued fraction expansions.  The principal
convergents from the continued fraction expansion of an
irrational real number $\alpha$ are defined by setting $p_{-2} = 0$,
$q_{-2} = 1$, $p_{-1} = 1$, $q_{-1}=0$, and then by the
recursive formulas
\begin{equation}
p_n = a_np_{n-1} + p_{n-2}\quad\text{and}\quad q_n = a_nq_{n-1} +
q_{n-2}\label{cfraceq1}
\end{equation}
for $n = 0, 1, 2, \dots $.  Of course $a_n = a_n(\alpha)$, $p_n =
p_n(\alpha)$, and $q_n = q_n(\alpha)$ depend
on $\alpha$, but to simplify notation we often suppress this dependence.

If $\alpha$ is an irrational point in $\R/\Z$, that is, $\alpha$ does not
belong to $\Q/\Z$, then the partial quotient
$a_0$ is not uniquely determined.  For our purposes it will be convenient
to set $a_0 = 0$ and so to identify $\alpha$ in
$\R/\Z$ with its coset representative in the open interval $(0,1)$.  Then
we also have $p_0 = 0$ and $q_0 = 1$.  We will
make use of the convergents and the intermediate convergents from the
continued fraction expansion of $\alpha$.
It will be convenient to organize these by defining
\begin{equation}
E_n = \left\{\frac{mp_{n-1} + p_{n-2}}{mq_{n-1} + q_{n-2}}: m = 1, 2,
\dots , a_n\right\}\label{cfraceq10}
\end{equation}
for $n = 1, 2, \dots $.
Each set $E_n$ contains $a_n$ distinct fractions, including the principal
convergent $p_n/q_n$.  The remaining fractions
(if any) indexed by $m = 1, 2, \dots , a_n-1$ are the intermediate
convergents to $\alpha$.  It is easy to check that
\begin{equation}
E_n = \big\{[0; a_1, a_2, \dots , a_{n-1}, m]: m = 1, 2, \dots,
a_n\big\}.\label{cfraceq18}
\end{equation}
Again the set $E_n = E_n(\alpha)$ depends on $\alpha$, but we often
suppress this dependence.

For each point $\beta$ in $\Q/\Z$ the function $\fb:\R/\Z\rar\R$ is a step
function taking at most three distinct values
on intervals of positive measure, and defined below by (\ref{fbdef1}) or
(\ref{fbdef2}).
If $\alpha$ is an irrational point in $\R/\Z$ then the value of
$\fb(\alpha)$ is also
determined by the convergents and intermediate convergents to $\alpha$.
The precise result is as follows.

\begin{theorem}\label{fbetacontfr}
Let $\alpha$ be an irrational point in $\R/\Z$, and for $n = 1, 2, \dots $
let $E_n$ be the collection of convergents and
intermediate convergents defined by {\rm (\ref{cfraceq10})}.  If $\beta\in\Q/\Z$
then $f_{\beta}(\alpha) \not= 0$ if and
only if $\beta$ belongs to $E_n$ for some $n = 1, 2, \dots $.  Moreover,
if $\beta$ belongs to $E_n$ then
\begin{equation}
f_{\beta}(\alpha) = (-1)^{n-1} q_{n-1}(\alpha).\label{fbcfprop1}
\end{equation}
\end{theorem}

\noindent Further identities involving partial sums of the functions $\fb$
are given in Section 3.

By combining Theorem \ref{fbetacontfr} with a convergence theorem for
martingale differences due to
 R.~F.~Gundy \cite[Theorem 2.1(a)]{gundy1966}, we obtain the following
metric theorem.

\begin{theorem}\label{cfserrep2}
Let $F:\R/\Z\rar \R\cup\{\pm\infty\}$ be a Borel measurable function that
is finite almost everywhere.  Then there
exist real numbers $\{c(\beta):\beta\in\Q/\Z\}$ such that
\begin{equation}
\lim_{N\rar\infty} \sum_{n=1}^N (-1)^{n-1} q_{n-1}(\alpha) \sum_{\beta\in
E_n(\alpha)} c(\beta) = F(\alpha)\label{sum3}
\end{equation}
for almost all irrational points $\alpha$ in $\R/\Z$.
\end{theorem}

\noindent The numbers $\{c(\beta):\beta\in\Q/\Z\}$ that occur in the
statement of Theorem \ref{cfserrep2} are not
uniquely determined by the measurable function $F$.  In particular, there
exist $c(\beta)$ that are not all zero but for
which the limit (\ref{sum3}) is zero for almost all $\alpha$.  We give an
example of this in Section 4.  An interesting
feature of Theorem \ref{cfserrep2} is that no assumption is made
concerning the integrability of $F$.
If we assume that the function $F$ is in $L^1(\R/\Z)$ then the conclusion
(\ref{sum3}) is much easier to prove.  In fact,
if $F$ is in $L^1(\R/\Z)$ then there is a unique choice of the numbers
$\{c(\beta):\beta\in\Q/\Z\}$ such that (\ref{sum3})
converges almost everywhere {\it and} in $L^1$-norm to the function $F$.
If $F$ is in $L^2(\R/\Z)$ then there is a unique
choice of the numbers $\{c(\beta):\beta\in\Q/\Z\}$ such that (\ref{sum3})
converges almost everywhere {\it and}
\begin{equation*}
\sum_{\beta\in\Q/\Z} c(\beta)^2 < \infty.
\end{equation*}
This follows from the fact (see Theorem \ref{comporththm}) that the
collection of functions $\{f_{\beta}:\beta\in\Q/\Z\}$
forms a complete orthonormal basis for $L^2(\R/\Z)$.

In Section 5 we assume that numbers $\{c(\beta):\beta\in\Q/\Z\}$ are given
and we consider the behavior of the corresponding
partial sums, such as occur on the left of (\ref{sum3}).  If the partial
sums are bounded in $L^1$-norm then it is an
immediate consequence of the martingale convergence theorem that the
partial sums converge almost everywhere.  We
report this as Theorem \ref{martthm1}.  If we assume that the map
$\beta\mapsto c(\beta)h(\beta)$ is bounded on $\Q/\Z$,
then we can draw further conclusions about the set of irrational points
$\alpha$ where the partial sums converge.

\begin{theorem}\label{martthm0}  Let $\{c(\beta):\beta\in\Q/\Z\}$ be a
collection of real numbers such that
$\beta\mapsto c(\beta)h(\beta)$ is bounded on $\Q/\Z$.  For each
irrational point $\alpha$ in $\R/\Z$ and positive integer
$Q$ let $M = M(\alpha, Q)$ and $N = N(\alpha, Q)$ be the unique positive
integers such that
\begin{equation*}
1 \leq M \leq a_N\quad\text{and}\quad Mq_{N-1} + q_{N-2} \leq Q <
(M+1)q_{N-1} + q_{N-2}.
\end{equation*}
Write $\cc$ for the subset of irrational points $\alpha$ in $\R/\Z$
such that
\begin{equation}
\lim_{Q\rar\infty} ~\sum_{n=1}^N (-1)^{n-1} q_{n-1}(\alpha)
	\sum_{\substack{\beta\in E_n(\alpha)\\h(\beta)\le Q}} c(\beta)\label{summ1}
\end{equation}
exists and is finite.  Write $\D$ for the subset of irrational points
$\alpha$ in $\R/\Z$ such that both
\begin{equation}
\liminf_{Q\rar\infty} ~\sum_{n=1}^N (-1)^{n-1} q_{n-1}(\alpha)
	\sum_{\substack{\beta\in E_n(\alpha)\\h(\beta)\le Q}} c(\beta) =
-\infty\label{summ2}
\end{equation}
and
\begin{equation}
\limsup_{Q\rar\infty} ~\sum_{n=1}^N (-1)^{n-1} q_{n-1}(\alpha)
	\sum_{\substack{\beta\in E_n(\alpha)\\h(\beta)\le Q}} c(\beta) =
+\infty.\label{summ3}
\end{equation}
Write $\E$ for the subset of irrational points $\alpha$ in $\R/\Z$ such that
\begin{equation*}
\sum_{n=1}^{\infty} q_{n-1}(\alpha)^2 \sum_{\beta\in E_n(\alpha)}
c(\beta)^2 < \infty.
\end{equation*}
Then we have
\begin{equation}
{\rm (i)}\ |\cc\cup\D| = 1,\quad {\rm (ii)}\ |\cc\setminus\E| = 0,
	\quad\text{and}\quad {\rm (iii)}\ |\E\setminus\cc| = 0.\label{sets0}
\end{equation}
\end{theorem}

\noindent  We note that the restriction $h(\beta) \le Q$ in (\ref{summ1}),
(\ref{summ2}), and (\ref{summ3}), effects only
the term for which $n = N$.  Obviously $N\rar\infty$ as $Q\rar\infty$, but
in a manner that depends on $\alpha$.

We would like to thank H.~L.~Montgomery for first calling our attention to
the functions $\fb$.


\section{A Complete System of Orthonormal Functions}

We will work with functions $g:\R/\Z\rar \C$ of bounded variation that
satisfy the condition
\begin{equation}
g(x) = \tfrac12g(x+) + \tfrac12g(x-)\label{norm1}
\end{equation}
at each point $x$.  When (\ref{norm1}) is satisfied we say that the
function $g$ is {\it normalized}.  The collection
of all normalized real (or complex) valued functions of bounded variation
on $\R/\Z$ is a
real (or complex) vector space.  An important example is the sawtooth
function $\psi:\R/\Z\rar \R$ defined by
\begin{equation*}
\psi(x) = \begin{cases} x-[x]-\tfrac12 & \text{if $x$ is not in $\Z$},\\
            0 & \text{if $x$ is in $\Z$},\end{cases}
\end{equation*}
where $[x]$ is the integer part of $x$.

For each positive integer $Q$ we define $\F_Q$ to be the finite set
$$\F_Q = \{\beta\in\Q/\Z: h(\beta) \le Q\}.$$
It follows that $\R/\Z\setminus\F_Q$ is the union of exactly
$$\bigl|\F_Q\bigr| = \sum_{q\le Q} \vphi(q)$$
component intervals, where $\vphi$ is the Euler $\vphi$-function.  Clearly
each component interval determines a unique
left hand endpoint $\beta_1$ in $\F_Q$ and a unique right hand endpoint
$\beta_2$ in $\F_Q$.  In this case it will be convenient to write
$I(\beta_1, \beta_2)$ for
the corresponding (open) component interval and $\oI(\beta_1, \beta_2)$
for its closure in $\R/\Z$.
We say that the elements of the ordered set $\{\beta_1, \beta_2\}$
are ${\it consecutive}$ points in $\F_Q$ if there exists a component
interval of the form $I(\beta_1, \beta_2)$ in
$\R/\Z\setminus\F_Q$.  More generally, if $\beta_1$ and $\beta_2$ are
points in $\Q/\Z$ we write $I(\beta_1, \beta_2)$
for the corresponding component interval whenever the elements of the
ordered set $\{\beta_1, \beta_2\}$ are consecutive
points in $\F_Q$ for some positive integer $Q$.
We note that the {\it normalized} characteristic function of the component
interval $I(\beta_1, \beta_2)$ is given by
$$\bigl|I(\beta_1, \beta_2)\bigr| + \psi(\beta_1 - x) + \psi(x - \beta_2),$$
where
$$\bigl|I(\beta_1, \beta_2)\bigr| = h(\beta_1)^{-1}h(\beta_2)^{-1}$$
is the Haar measure of $I(\beta_1, \beta_2)$.

Now suppose that $\beta$ is a nonzero point in $\Q/\Z$ such that $h(\beta)
= Q$.  Then there exists a unique point
$\beta'$ in $\F_Q$ such that $\{\beta', \beta\}$ are consecutive points in
$\F_Q$, and there exists a unique point $\beta''$
in $\F_Q$ such that $\{\beta, \beta''\}$ are consecutive points in $\F_Q$.
 Thus we have two well defined maps
$\beta\mapsto\beta'$ and $\beta\mapsto\beta''$ from $\Q/\Z\setminus\{0\}$
into $\Q/\Z$.  It is easy to verify that these maps
are both surjective.  And they satisfy the basic identities
\begin{equation}
h(\beta) = h(\beta') + h(\beta''),\quad \gcd\{h(\beta'),h(\beta)\} =
1,\quad \gcd\{h(\beta),h(\beta'')\} = 1,\label{hp1}
\end{equation}
and also
\begin{equation}
h(\beta - \beta') = h(\beta)h(\beta')\quad\text{and}\quad h(\beta''-\beta)
= h(\beta)h(\beta'').\label{hp2}
\end{equation}
We will often use the fact that if $\beta$ is a nonzero point in $\Q/\Z$
then $I(\beta',\beta'')$ is a component interval of $\R/\Z\setminus\F_q$
for all $q$ such that
$\max\{h(\beta'),h(\beta'')\} \le q < h(\beta)$.
Of course the remarks and notation we have developed here reflect well
known properties of Farey fractions (see \cite{hardy1979}
or \cite{khin1964}), but modified slightly to account for our working in
the group $\Q/\Z$.
The following result also follows easily from basic properties of Farey
fractions.

\begin{lemma}\label{farintlem}
Suppose that $\beta$ and $\gamma$ are distinct nonzero points in $\Q/\Z$.
If $h(\beta) \le h(\gamma)$ then exactly one of the
following holds:
$$I(\gamma',\gamma'')\subseteq I(\beta',\beta)\quad\text{or}
 \quad I(\gamma',\gamma'')\subseteq I(\beta,\beta'')\quad\text{or}\quad
I(\beta',\beta'')\cap I(\gamma',\gamma'') = \emptyset.$$
\end{lemma}

If $g_1(x)$ and $g_2(x)$ are functions in $L^2\left(\R/\Z\right)$ we write
$$\langle g_1,g_2 \rangle = \int_{\R/\Z} g_1(x)\overline{g_2(x)}\
\dx\quad\text{and}\quad \|g_1\|_2
	= \Big\{\int_{\R/\Z} |g_1(x)|^2\ \dx\Big\}^{1/2}$$
for their inner product and norm, respectively.

For each point $\beta$ in $\Q/\Z$ we define a normalized function
$f_{\beta}:\R/\Z\rightarrow \R$ of bounded variation as follows.  If
$\beta = 0$ we
set $f_\beta(x) = 1$, and if $\beta \not= 0$ we set
\begin{equation}
f_{\beta}(x) = h(\beta)\psi(x - \beta) - h(\beta')\psi(x - \beta') -
h(\beta'')\psi(x - \beta'').\label{fbdef1}
\end{equation}
As the integral of $\psi$ over $\R/\Z$ is $0$, it follows immediately that
\begin{equation}
\int_{\R/\Z}\fb(x)\ \dx = \begin{cases} 1 &\text{if $\beta = 0$},
	\\ 0 &\text{if $\beta \not= 0$}.\end{cases}\label{fbintegral}
\end{equation}
If $h(\beta) = q \ge 2$, $h(\beta') = q'$, and $h(\beta'') = q''$, then a
useful alternative definition of $\fb$ is given by
\begin{align}
\fb(x) & = \begin{cases}q' &\text{ if $x\in I(\beta',\beta)$},\\
	-q'' &\text{ if $x\in I(\beta,\beta'')$},\\
	\tfrac12 q' &\text{ if $x=\beta'$},\\
	-\tfrac12 q'' &\text{ if $x=\beta''$},\\
	\tfrac12 (q'-q'') &\text{ if $x=\beta$},\\
	0 &\text{ if $x\notin\oI(\beta',\beta'')$}.\end{cases}\label{fbdef2}
\end{align}
It is obvious from (\ref{fbdef2}) that for $\beta\not= 0$ the function
$\fb$ is supported on the closed
set $\oI(\beta',\beta'')$.  Also, using (\ref{fbdef2}) we find that
\begin{equation}
\|\fb\|_2^2 = \int_{\R/\Z}\fb(x)^2\ \dx = \frac{(q')^2}{q' q} +
\frac{(q'')^2}{q q''} = 1.\label{fbnorm}
\end{equation}
Thus each function $\fb$ has norm $1$ and $\langle f_0, \fb\rangle = 0$
for $\beta \not= 0$.

Now suppose that $\beta$ and $\gamma$ are distinct nonzero points of
$\Q/\Z$.  Without loss of generality
we may assume that $h(\beta) \le h(\gamma)$. In view of Lemma
\ref{farintlem} there are three cases to
consider. If $I(\gamma',\gamma'')\subseteq I(\beta', \beta)$ then
\begin{equation*}
\langle f_{\beta},f_{\gamma}\rangle = \int_{\left(\gamma',\gamma''\right)}
f_{\beta}(x)f_{\gamma}(x)\ \dx
	= f_{\beta}(\gamma)\int_{\left(\gamma',\gamma''\right)} f_{\gamma}(x)\
\dx = 0.
\end{equation*}
The other cases lead to the same conclusion in a similar manner.
This shows that the collection of functions $\{\fb:\beta\in\Q/\Z\}$ forms
an orthonormal subset of $L^2(\R/\Z)$.

It remains now to show that the functions $\{\fb:\beta\in\Q/\Z\}$ form a
complete orthonormal basis for $L^2(\R/\Z)$.
Toward this end let $Q$ be a positive integer and define
$K_Q:\R/\Z\times\R/\Z\rightarrow \R$ by
\begin{equation}
K_Q(x,y) =\sum_{\beta\in\F_Q}f_{\beta}(x)f_{\beta}(y).\label{KQeq8}
\end{equation}
Note that for each $x$ the function $y\mapsto K_Q(x,y)$ is normalized and
that for each $y$ the function
$x\mapsto K_Q(x,y)$ is normalized.  For each $Q\geq 2$ we define a
function $J_Q:\R/\Z\times\R/\Z\rightarrow \R$ by
$$J_Q(x,y) =\sum_{h(\beta)=Q}f_{\beta}(x)f_{\beta}(y).$$
For $Q\geq 2$ it is clear that
\begin{equation}
K_Q(x,y) = K_{Q-1}(x,y) + J_Q(x,y).\label{KQeq9}
\end{equation}
We also define a map
$$\sigma:\left(\R/\Z\setminus\Q/\Z\right)\times\left(\R/\Z\setminus\Q/\Z\right)\rar\{1,
2, \dots \}\cup\{\infty\}$$
as follows: if $x=y$ then $\sigma(x,y) = \infty$, and if $x$ and $y$
are distinct irrational points in $\R/\Z$ we define $\sigma(x,y)$ to be
the smallest positive integer $Q$ such
that $x$ and $y$ are not in the same component interval of
$\R/\Z\setminus\F_Q$.

\begin{lemma}\label{KQeval}
Let $x$ and $y$ be points in $\R/\Z\setminus\F_Q$.
If $x$ and $y$ belong to the same component interval
$I(\gamma_1,\gamma_2)$ of $\R/\Z\setminus\F_Q$ then we have
$$K_Q(x,y) = h(\gamma_1)h(\gamma_2).$$
If $x$ and $y$ belong to distinct component intervals of
$\R/\Z\setminus\F_Q$ then
$$K_Q(x,y) = 0.$$
\end{lemma}

\begin{proof}
If $x$ and $y$ are irrational and $\sigma(x,y) \le Q-1$ then it is easily
seen that
$J_Q(x,y) = 0$.  If $\sigma(x,y)= Q$ then there exists a unique
element $\beta$ in $\F_Q\setminus\F_{Q-1}$ such that (after renaming $x$
and $y$ if necessary)
$$x\in I(\beta',\beta)\quad\text{and}\quad y\in I(\beta, \beta'').$$
It follows that $h(\beta) = Q$ and that $J_Q(x,y) =
-h(\beta')h(\beta'')$.  If $Q+1 \le \sigma(x,y)$ then there exists a pair of
consecutive points $\{\gamma_1, \gamma_2\}$ in $\F_Q$ such that both $x$
and $y$ belong to the component $I(\gamma_1, \gamma_2)$.
In this case we find that
$$J_Q(x,y) = \begin{cases} h(\gamma_2)^2 &\text{if $h(\gamma_1) = Q$},
	\\ h(\gamma_1)^2 &\text{if $h(\gamma_2) = Q$},\\ 0
&\text{otherwise}.\end{cases}$$
Next we use this information about $J_Q(x,y)$ to determine $K_Q(x,y)$.

We argue by induction on $Q$.  The case $Q=1$ is trivial, so we
assume that $Q\geq 2$ and that the assertion of the lemma holds for
$K_{Q-1}(x,y)$.
Now when $x$ and $y$ are irrational there are three cases to
consider.

If $\sigma(x,y) \le Q-1$ then $J_Q(x,y) = 0$ and $K_{Q-1}(x,y)= 0$
by the inductive hypothesis.  Hence $K_Q(x,y) = 0$ by (\ref{KQeq9}).

If $\sigma(x,y) = Q$ then there exists a unique point $\beta$ in
$\F_Q\setminus\F_{Q-1}$ such
that (after renaming $x$ and $y$ if necessary)
$$x\in I(\beta',\beta)\quad\text{and}\quad y\in I(\beta, \beta'').$$
We conclude that
$$J_Q(x,y) = -h(\beta')h(\beta''),$$
and by the inductive hypothesis
$$K_{Q-1}(x,y) = h(\beta')h(\beta'').$$
Again we find that $K_Q(x,y) = 0$ by (\ref{KQeq9}).

Finally, if $Q+1 \le \sigma(x,y)$ then there exists a pair of
consecutive points $\{\gamma_1, \gamma_2\}$ in $\F_Q$ such that both $x$
and $y$ belong to the component $I(\gamma_1, \gamma_2)$.
If $h(\gamma_1) \leq Q-1$ and $h(\gamma_2) \le Q-1$ then $J_Q(x,y) = 0$ and
$$K_{Q-1}(x,y) = h(\gamma_1)h(\gamma_2)$$
by the inductive hypothesis.  If $h(\gamma_1) = Q$ then $J_Q(x,y) =
h(\gamma_2)^2$.  It follows that $\gamma_2 = \gamma_1''$
and therefore $\{\gamma_1', \gamma_1''\}$ are consecutive points in
$\F_{Q-1}$.  Thus we find that
$$K_{Q-1}(x,y) = h(\gamma_1')h(\gamma_1'')$$
by the inductive hypothesis, and we conclude that
$$K_Q(x,y) = h(\gamma_1')h(\gamma_1'') + h(\gamma_1'')^2 =
h(\gamma_1)h(\gamma_1'') =  h(\gamma_1)h(\gamma_2).$$
If $h(\gamma_2) = Q$ the argument is essentially the same.  This proves
the lemma
when $x$ and $y$ are irrational.  However, comparing
(\ref{fbdef2}) and (\ref{KQeq8}) it is easy to see that the
functions $x\mapsto K_Q(x,y)$ and $y\mapsto K_Q(x,y)$ are constant
on the interior of all component intervals of $\F_Q$, so the result
of the lemma extends immediately to all points $x$ and $y$ in
$\R/\Z\setminus\F_Q$.
\end{proof}

\begin{lemma}\label{KQleb1}
Let $g:\R/\Z\rightarrow \R$ be an integrable function.  Then for
almost all points $x$ in $\R/\Z$ we have
$$\lim_{Q\rightarrow \infty} \int_{\R/\Z} g(y) K_Q(x,y)\ \dy = g(x).$$
\end{lemma}

\begin{proof}
By the Lebesgue density theorem we have
\begin{equation}
\lim_{z\rightarrow x} (z - x)^{-1}\int_x^z |g(y) - g(x)|\ \dy =
0\label{KQeq10}
\end{equation}
for almost all $x$ in $\R$.  Assume that $x$ is an irrational real number
such that (\ref{KQeq10}) holds.  For each $Q$ let
$\{\beta_Q, \gamma_Q\}$ be consecutive points in $\F_Q$ such that $x$
belongs to $I(\beta_Q, \gamma_Q)$.  Let $\beta_Q$ and
$\gamma_Q$ be coset representatives such that $\beta_Q < x < \gamma_Q$ and
$\gamma_Q - \beta_Q \le Q^{-1}$.  Then by Lemma
\ref{KQeval} we have
\begin{align*}
\left|\int_0^1 g(y) K_Q(x,y)\ \dy - g(x)\right|
	&=\left|(\gamma_Q - \beta_Q)^{-1}\int_{\beta_Q}^{\gamma_Q}\left(g(y) -
g(x)\right)\ \dy\right|\\
	&\le (\gamma_Q - \beta_Q)^{-1}\int_{\beta_Q}^x \left|g(y) - g(x)\right|\
\dy\\
	&\quad + (\gamma_Q - \beta_Q)^{-1}\int_x^{\gamma_Q} \left|g(y) -
g(x)\right|\ \dy\\
	&\le (x - \beta_Q)^{-1}\int_{\beta_Q}^x \left|g(y) - g(x)\right|\ \dy\\
	&\quad + (\gamma_Q - x)^{-1}\int_x^{\gamma_Q} \left|g(y) - g(x)\right|\ \dy.
\end{align*}
Because
$$\lim_{Q\rightarrow\infty} \beta_Q =
x\quad\text{and}\quad\lim_{Q\rightarrow\infty} \gamma_Q = x,$$
the result follows from (\ref{KQeq10}).
\end{proof}

\begin{theorem}\label{comporththm}
The collection of functions $\{\fb :\beta\in\Q/\Z\}$ forms a complete,
orthonormal basis for $L^2(\R/\Z)$.
\end{theorem}

\begin{proof}
Suppose that $g(x)$ is in $L^2(\R/\Z)$ and $g(x)$ is orthogonal to each
function $\fb(x)$.  That is, we suppose that
$$\langle g,\fb\rangle = \int_{\R/\Z}g(y)\fb(y)\ \dy = 0$$
for each $\beta$ in $\Q/\Z$.  Then we have
\begin{equation}
\int_{\R/\Z} g(y) K_Q(x,y)\ \dy= \sum_{\beta\in\F_Q}\langle g,\fb\rangle
\fb(x) =0\label{KQeq12}
\end{equation}
for all points $x$ in $\R/\Z$.  Letting $Q\rightarrow\infty$ in
(\ref{KQeq12}) and applying Lemma \ref{KQleb1}, it
follows that $g(x)$ is $0$ in $L^2(\R/\Z)$.  This shows that the
collection $\{\fb\}$ is complete in $L^2(\R/\Z)$.
\end{proof}


\section{The Continued Fraction Interpretation}

We return to consideration of the principal convergents and intermediate
convergents from the continued fraction expansion
of an irrational real number $\alpha$.  It will be convenient to assume
that $0 < \alpha < 1$ and to write $\ff_Q$ for the
set of Farey fractions in $[0,1]$ of order $Q$.  As is well known, the
convergents and intermediate convergents to
$\alpha$ satisfy the following inequalities.  If $n$ is an odd positive
integer we have
\begin{equation}
\frac{p_{n-1}}{q_{n-1}} < \alpha < \frac{a_np_{n-1} + p_{n-2}}{a_nq_{n-1}
+ q_{n-2}}< \dots < \frac{2p_{n-1} + p_{n-2}}{2q_{n-1}
	+ q_{n-2}} < \frac{p_{n-1} + p_{n-2}}{q_{n-1}+q_{n-2}}\label{cfraceq11}
\end{equation}
and if $n$ is an even positive integer then
\begin{equation}
\frac{p_{n-1} + p_{n-2}}{q_{n-1} +
q_{n-2}}<\frac{2p_{n-1}+p_{n-2}}{2q_{n-1}+q_{n-2}}<\dots<\frac{a_np_{n-1}+p_{n-2}}{a_nq_{n-1}
	+ q_{n-2}} < \alpha < \frac{p_{n-1}}{q_{n-1}}.\label{cfraceq12}
\end{equation}

Next we observe that for each positive integer $Q$ there exists a unique
pair of positive integers $M$ and $N$ such that
\begin{equation}
1 \leq M \leq a_N\quad\text{and}\quad Mq_{N-1} + q_{N-2} \leq Q <
(M+1)q_{N-1} + q_{N-2}.\label{cfraceq15}
\end{equation}
If $N$ is odd we have
\begin{equation}
\frac{p_{N-1}}{q_{N-1}} < \alpha < \frac{Mp_{N-1} +
p_{N-2}}{Mq_{N-1}+q_{N-2}}\label{cfraceq16}
\end{equation}
and if $N$ is even then
\begin{equation}
\frac{Mp_{N-1} + p_{N-2}}{Mq_{N-1} + q_{N-2}}<\alpha <
\frac{p_{N-1}}{q_{N-1}}.\label{cfraceq17}
\end{equation}
Equations (\ref{cfraceq16}) and (\ref{cfraceq17}) determine the unique
open Farey interval in $[0,1]\setminus \ff_Q$ that
contains $\alpha$.

\begin{lemma}\label{farey0}
Let $\alpha$ be an irrational point in $\R/\Z$ and let
\begin{equation*}
\beta = \frac{mp_{n-1} + p_{n-2}}{mq_{n-1}+q_{n-2}},\quad\text{where}\quad
1 \leq m \leq a_n,
\end{equation*}
be a nonzero point in $E_n(\alpha)$ for some positive integer $n$.  If $n$
is odd then
\begin{equation}
\beta' = \frac{p_{n-1}}{q_{n-1}},\quad\beta'' = \frac{(m-1)p_{n-1} +
p_{n-2}}{(m-1)q_{n-1} + q_{n-2}},
	\quad\text{and}\quad \alpha\in I(\beta',\beta).\label{pr1}
\end{equation}
If $n$ is even then
\begin{equation}
\beta' = \frac{(m-1)p_{n-1} + p_{n-2}}{(m-1)q_{n-1} +
q_{n-2}},\quad\beta'' = \frac{p_{n-1}}{q_{n-1}},
	\quad\text{and}\quad \alpha\in I(\beta,\beta'').\label{pr2}
\end{equation}
\end{lemma}

\begin{proof}  If $n$ is an odd positive integer then it follows from
properties of Farey fractions that the three fractions
\begin{equation}
\frac{p_{n-1}}{q_{n-1}} < \frac{mp_{n-1} + p_{n-2}}{mq_{n-1} + q_{n-2}} <
\frac{(m-1)p_{n-1} +
	p_{n-2}}{(m-1)q_{n-1}+q_{n-2}}\label{cfraceq13}
\end{equation}
are consecutive points of $\ff_Q$ for $Q = mq_{n-1} + q_{n-2}$.  (Note
that $(m,n) \not= (1,1)$ because $\beta \not= 0$.)
This verifies the identities for $\beta'$ and $\beta''$ in (\ref{pr1}).
Then it follows from (\ref{cfraceq11}) that
$\alpha$ belongs to $I(\beta',\beta)$.  Similarly, if $n$ is an even
positive integer then the three fractions
\begin{equation}
\frac{(m-1)p_{n-1} + p_{n-2}}{(m-1)q_{n-1}+q_{n-2}} <
\frac{mp_{n-1}+p_{n-2}}{mq_{n-1} +
	q_{n-2}}<\frac{p_{n-1}}{q_{n-1}}\label{cfraceq14}
\end{equation}
are consecutive points of $\ff_Q$ for $Q = mq_{n-1} + q_{n-2}$.  The
assertions in (\ref{pr2}) follow as in the previous case.
\end{proof}

The functions $\{\fb\}$ were initially defined by (\ref{fbdef1}) and
(\ref{fbdef2}).  We now show that the value
of $\fb(\alpha)$ depends in a simple way on the convergents and
intermediate convergents to $\alpha$.

\begin{proof}[Proof of Theorem \ref{fbetacontfr}]
If $\beta = 0$ then $\beta$ belongs to $E_1(\alpha)$  and
(\ref{fbcfprop1}) is obvious.  Assume that $\beta\not= 0$ belongs
to $E_n$ and that
\begin{equation*}
\beta = \frac{mp_{n-1} + p_{n-2}}{mq_{n-1}+q_{n-2}}\quad\text{where}\quad
1 \leq m \leq a_n.
\end{equation*}
If $n$ is odd then it follows from (\ref{pr1}) that $\alpha$ belongs to
the component interval $I(\beta',\beta)$.  From
the definition of $f_{\beta}$ we conclude that
$$f_{\beta}(\alpha) = q_{n-1} = (-1)^{n-1}q_{n-1}.$$
Similarly, if $n$ is even then (\ref{pr2}) implies that $\alpha$ belongs
to the component interval $I(\beta,\beta'')$.
In this case we find that
$$f_{\beta}(\alpha) = -q_{n-1} = (-1)^{n-1}q_{n-1}.$$

Now assume that $f_{\beta}(\alpha) \not= 0$.  If $\beta = 0$ then $\beta$
belongs to $E_1$.  Otherwise we have either
\begin{equation}
\alpha\in I(\beta',\beta)\quad\text{or}\quad \alpha\in
I(\beta,\beta'').\label{fbcfprop3}
\end{equation}
Write $h(\beta)=Q$ and as in (\ref{cfraceq15}) let $M$ and $N$ be the
unique positive integers such that
$$1 \le M \le a_N\quad\text{and}\quad Mq_{N-1} + q_{N-2} \le Q <
(M+1)q_{N-1} + q_{N-2}.$$
If $N$ is odd then (\ref{cfraceq16}) and (\ref{fbcfprop3}) imply that
\begin{equation}
\beta = \frac{Mp_{N-1} + p_{N-2}}{Mq_{N-1} + q_{N-2}},\label{fbcfprop4}
\end{equation}
and this shows that $\beta$ belongs to $E_N$.  Similarly, if $N$ is even
then (\ref{cfraceq17}) and (\ref{fbcfprop3}) imply
that (\ref{fbcfprop4}) holds, and again we conclude that $\beta$ belongs
to $E_N$.
\end{proof}

Certain partial sums involving the functions $\fb$ also have a natural
Diophantine interpretation.  For each positive integer $Q$ we define two
functions
$$L_Q:\R/\Z\setminus \F_Q \rightarrow \{1, 2, \dots , Q\}\quad\text{and}
	\quad R_Q:\R/\Z\setminus \F_Q \rightarrow \{1, 2, \dots , Q\}.$$
If $\alpha$ is a point in $\R/\Z\setminus \F_Q$ then there exists a unique
pair $\{\gamma_1, \gamma_2\}$ of consecutive
points in $\F_Q$ such that $\alpha$ belongs to $I(\gamma_1, \gamma_2)$.
We define $L_Q(\alpha) = h(\gamma_1)$ and
$R_Q(\alpha) =h(\gamma_2)$.  From (\ref{KQeq8}) and Lemma \ref{KQeval} we
obtain the identity
\begin{equation}
1+\sum_{2\leq h(\beta)\le Q} f_{\beta}(\alpha)^2 = L_Q(\alpha)R_Q(\alpha)
\label{LReq4}
\end{equation}
for $\alpha$ in $\R/\Z\setminus \F_Q$.  We now establish some further
identities of this sort.

\begin{lemma}\label{LQRQeval}
If $\alpha\in\R/\Z\setminus \F_Q$ then
\begin{equation}
2+\sum_{2\le h(\beta)\le Q} |f_{\beta}(\alpha)| =
R_Q(\alpha)+L_Q(\alpha),\label{LReq5}
\end{equation}
and
\begin{equation}
\sum_{2\le h(\beta)\le Q} f_{\beta}(\alpha) = R_Q(\alpha) -
L_Q(\alpha).\label{LReq6}
\end{equation}
\end{lemma}

\begin{proof}
As in the proof of Lemma \ref{KQeval}, it suffices to establish
(\ref{LReq5}) and (\ref{LReq6}) for $\alpha$ irrational.  For
$n = 1, 2, \dots $ let $E_n$ be the collection of convergents and
intermediate convergents defined by (\ref{cfraceq10}).  Let $M$ and $N$ be
the unique positive integers
such that (\ref{cfraceq15}) holds.  If $N$ is odd then from
(\ref{cfraceq16}) we find that
\begin{equation*}
\{\gamma_1,\gamma_2\} = \Bigg\{\frac{p_{N-1}}{q_{N-1}}, \frac{Mp_{N-1} +
p_{N-2}}{Mq_{N-1}+q_{N-2}}\Bigg\}
\end{equation*}
are consecutive points in $\F_Q$ such that $\alpha$ belongs to
$I(\gamma_1,\gamma_2)$.  This implies that
\begin{equation}
L_Q(\alpha) = q_{N-1}\quad\text{and}\quad R_Q(\alpha) = Mq_{N-1} +
q_{N-2}.\label{LReq8}
\end{equation}
If $N$ is even then (\ref{cfraceq17}) implies that
\begin{equation*}
\{\gamma_1,\gamma_2\} = \Bigg\{\frac{Mp_{N-1} + p_{N-2}}{Mq_{N-1} +
q_{N-2}},\frac{p_{N-1}}{q_{N-1}}\Bigg\}
\end{equation*}
are consecutive points in $\F_Q$ such that $\alpha$ belongs to
$I(\gamma_1,\gamma_2)$.  In this case we
conclude that
\begin{equation}
L_Q(\alpha) = Mq_{N-1} + q_{N-2}\quad\text{and}\quad R_Q(\alpha) =
q_{N-1}.\label{LReq10}
\end{equation}
For each $m\in\{1, 2, \dots , a_N\}$ write
$$\beta_m = \frac{mp_{N-1} + p_{N-2}}{mq_{N-1}+q_{N-2}}.$$
Using (\ref{fbcfprop1}), (\ref{LReq8}), and (\ref{LReq10}) we find that
\begin{align}
\sum_{h(\beta)\le Q} |f_{\beta}(\alpha)|&= \sum_{n=1}^{N-1} \sum_{\beta\in
E_n} |f_{\beta}(\alpha)|
	+ \sum_{m=1}^M |f_{\beta_m}(\alpha)|\nonumber\\
    &= \sum_{n=1}^{N-1} \sum_{\beta\in E_n} q_{n-1} + \sum_{m=1}^M
q_{N-1}\nonumber\\
    &= \sum_{n=1}^{N-1} a_nq_{n-1} + Mq_{N-1}\label{LReq11}\\
    &= \sum_{n=1}^{N-1} (q_n - q_{n-2}) + Mq_{N-1}\nonumber\\
    &= q_{N-1} + Mq_{N-1} + q_{N-2} - 1\nonumber\\
    &= R_Q(\alpha) + L_Q(\alpha) - 1.\nonumber
\end{align}
As $f_0(\alpha) = 1$, it is clear that (\ref{LReq11}) is equivalent
to (\ref{LReq5}).

In a similar manner, using (\ref{LReq8}) and (\ref{LReq10}) we get
\begin{align}
\sum_{h(\beta)\le Q} f_{\beta}(\alpha)&= \sum_{n=1}^{N-1} \sum_{\beta\in
E_n} f_{\beta}(\alpha)
	+ \sum_{m=1}^Mf_{\beta_m}(\alpha)\nonumber\\
    &= \sum_{n=1}^{N-1} \sum_{\beta\in E_n} (-1)^{n-1}q_{n-1} +
\sum_{m=1}^M(-1)^{N-1}q_{N-1}\nonumber\\
    &= \sum_{n=1}^{N-1} (-1)^{n-1}a_nq_{n-1} + (-1)^{N-1}Mq_{N-1}\nonumber\\
    &= \sum_{n=1}^{N-1} (-1)^{n-1}(q_n - q_{n-2}) +
(-1)^{N-1}Mq_{N-1}\nonumber\\
    &= (-1)^{N-1}\{Mq_{N-1} + q_{N-2} - q_{N-1}\} +1\nonumber\\
    &= R_Q(\alpha) - L_Q(\alpha) + 1,\nonumber
\end{align}
which proves (\ref{LReq6}).
\end{proof}
\begin{corollary}
If $\alpha$ belongs to $\R/\Z\setminus \F_Q$ then
\begin{equation*}
1+\sum_{2\le h(\beta)\le Q} f_{\beta}^+(\alpha) =
R_Q(\alpha)\quad\text{and}\quad 1
	+ \sum_{2\le h(\beta)\le Q}f_{\beta}^-(\alpha)= L_Q(\alpha).
\end{equation*}
\end{corollary}
\begin{proof}  Identity (\ref{LReq5}) can be written as
\begin{equation}
\Big\{1 + \sum_{2\le h(\beta)\le Q} f_{\beta}^+(\alpha)\Big\}
	+ \Big\{1 + \sum_{2\le h(\beta)\le Q} f_{\beta}^-(\alpha)\Big\} =
R_Q(\alpha) +L_Q(\alpha),\label{LReq14}
\end{equation}
and identity (\ref{LReq6}) can be written as
\begin{equation}
\Big\{1 + \sum_{2\le h(\beta)\le Q} f_{\beta}^+(\alpha)\Big\}
	- \Big\{1 + \sum_{2\le h(\beta)\le Q} f_{\beta}^-(\alpha)\Big\} =
R_Q(\alpha) - L_Q(\alpha).\label{LReq15}
\end{equation}
The statement of the corollary now plainly follows from
(\ref{LReq14}) and (\ref{LReq15}).
\end{proof}

The argument used to prove Theorem \ref{fbetacontfr} can be applied
to other functions indexed by points $\beta$ in $\Q/\Z$ and supported on
$\overline{I}(\beta', \beta'')$.  As an example we define a further
collection of
real valued functions $\left\{\chi_\beta:\beta\in\Q/\Z\right\}$ with
domain $\R/\Z$ as follows.  For $\beta = 0$ we set $\chi_0(x) = 1$ for all
$x$ in $\R/\Z$.  Then for $\beta\not= 0$ we set
\begin{equation*}
\chi_\beta(x)=\begin{cases}1&\text{if $x\in I(\beta',\beta'')$},\\
	\tfrac12 &\text{if $x = \beta'$ or $x = \beta''$},\\
	0&\text{if $x\notin \oI(\beta',\beta'')$}.\end{cases}
\end{equation*}
For each nonzero point $\beta$ in $\Q/\Z$ the function $\chi_{\beta}(x)$
is the
normalized characteristic function of the component interval
$I(\beta',\beta'').$  If $\beta_1$ and $\beta_2$ are distinct
points in $\Q/\Z$ with $h(\beta_1) = h(\beta_2)\geq 2$, then the open
component intervals $I(\beta_1',\beta_1'')$ and
$I(\beta_2',\beta_2'')$ are disjoint.  Thus for each positive integer $q$
the normalized characteristic function of the subset
\begin{equation*}
\bigcup_{h(\beta) = q} I(\beta',\beta'').
\end{equation*}
is the function $X_q:\R/\Z\rar\R$ defined by
\begin{equation}
X_q(x)=\sum_{h(\beta)=q}\chi_{\beta}(x).\label{Xqdef}
\end{equation}

\begin{theorem}\label{cfracchi}
Let $\alpha$ be an irrational point in $\R/\Z$, and for $n = 1, 2, \dots $
let $E_n$ be the collection of convergents and
intermediate convergents defined by {\rm (\ref{cfraceq10})}.  If $\beta$ is in
$\Q/\Z$ then
$\chi_{\beta}(\alpha) = 1$ if and only if $\beta$ belongs to $E_n$
for some $n = 1, 2, \dots $.  Moreover, the sum
\begin{equation*}
\sum_{h(\beta)\le Q} \chi_{\beta}(\alpha) = \sum_{q=1}^Q X_q(\alpha)
\end{equation*}
is exactly the number of convergents and intermediate convergents to
$\alpha$ with height less than or equal to $Q$.
\end{theorem}

\begin{proof}
The first assertion of the corollary follows as in the proof of Theorem
\ref{fbetacontfr}.  For the second assertion let
$M$ and $N$ be the unique positive integers such that
$$1 \le M \le a_N\quad\text{and}\quad Mq_{N-1} + q_{N-2} \le Q <
(M+1)q_{N-1} + q_{N-2},$$
and for each $m\in\{1, 2, \dots , a_N\}$ write
$$\beta_m =\frac{mp_{N-1} + p_{N-2}}{mq_{N-1} + q_{N-2}}.$$
Then we have
\begin{eqnarray*}
\sum_{h(\beta)\le Q} \chi_{\beta}(\alpha)&=&
\sum_{n=1}^{N-1}\sum_{\beta\in E_n} \chi_{\beta}(\alpha)
	+ \sum_{m=1}^M \chi_{\beta_m}(\alpha)\\
&=& \sum_{n=1}^{N-1} \sum_{\beta\in E_n} 1 + M\\
&=& \sum_{n=1}^{N-1} a_n + M.
\end{eqnarray*}
Plainly, this is the number of convergents and intermediate convergents to
$\alpha$ with height less than or equal to $Q$.
\end{proof}

Let $\qq$ be a subset of positive integers.  For each irrational point
$\alpha$ in $\R/\Z$ write
\begin{equation*}
\D(\alpha) = \{m q_{n-1} + q_{n-2}: 1 \le m \le a_n\ \text{and}\ 1 \le n \}
\end{equation*}
for the set of denominators from the collection of convergents and
intermediate convergents to $\alpha$.  Arguing as in the
proof of Theorem \ref{cfracchi}, we find that
\begin{equation}
\sum_{q\in\qq} X_q(\alpha) = \bigl|\D(\alpha)\cap \qq\bigr|.\label{exsubq0}
\end{equation}
If the set $\qq$ is such that the integral
\begin{equation}
\int_{\R/\Z} \Big\{\sum_{q\in\qq} X_q(x)\Big\}\ \dx \label{exsubq1}
\end{equation}
is finite then the integrand is finite for almost all $x$, and therefore
(\ref{exsubq0}) is finite for almost all irrational
points $\alpha$ in $\R/\Z$.  We conjecture that if the integral
(\ref{exsubq1}) is infinite then (\ref{exsubq0}) is
infinite for almost all irrational points $\alpha$.  The situation is
clarified by the following simple estimate.

\begin{lemma}\label{exsubq}
For each integer $q \ge 2$ we have
\begin{equation}
\int_{\R/\Z} X_q(x)\ \dx = \frac{2 \vphi(q)}{q^2}\Big\{\log q + \sum_{p|q}
\frac{\log p}{p-1}
	+ c_0\Big\} + O\Bigl(\frac{\log\log q}{q^2}\Bigr),\label{exsubq2}
\end{equation}
where $c_0$ is Euler's constant, and the sum on the right of
{\rm (\ref{exsubq2})} is over prime numbers $p$ that divide $q$.
\end{lemma}

\begin{proof}  Suppose that $\beta = a/q$, where $1 \le a < q$ and $(a,q)
= 1$.  Write $\oa$ for the unique
integer such that $1 \le \oa < q$ and $a\oa \equiv 1 \pmod q$.  We find
that $h(\beta') = \oa$ and
$h(\beta'') = q - \oa$, and therefore
\begin{align}
\int_{\R/\Z} X_q(x)\ \dx &= \sum_{\substack{\beta\in\Q/\Z\\h(\beta) = q}}
\frac{1}{h(\beta')h(\beta'')}
	 = \sum_{\substack{a=1\\(a,q)=1}}^q \frac{1}{\oa(q - \oa)}\nonumber\\
	&= \sum_{\substack{a=1\\(a,q) = 1}}^q \Bigl(\frac{1}{\oa q} + \frac{1}{(q
- \oa) q}\Bigr)\label{ds1}\\
	&= \sum_{\substack{a=1\\(a,q) = 1}}^q \frac{2}{a q}.\nonumber
\end{align}
Then using M\"obius inversion and well known estimates we have
\begin{align}
\sum_{\substack{a=1\\(a,q) = 1}}^q \frac{1}{a} &= \sum_{a=1}^q \frac{1}{a}
\sum_{\substack{d|q\\d|a}} \mu(d)
	= \sum_{d|q} \mu(d) \sum_{\substack{a=1 \\ d|a}}^q \frac{1}{a}
	= \sum_{d|q} \frac{\mu(d)}{d} \sum_{b=1}^{q/d} \frac{1}{b}\nonumber\\
	&= \sum_{d|q} \frac{\mu(d)}{d} \Big\{\log q - \log d + c_0 + \frac{d}{2q}
		+ O\Bigl(\frac{d^2}{q^2}\Bigr)\Big\}\label{ds2}\\
	&=  \frac{\vphi(q) \log q}{q} - \sum_{d|q} \frac{\mu(d) \log d}{d} +
\frac{c_0 \vphi(q)}{q}
		+ O\Bigl(\frac{\log\log q}{q}\Bigr).\nonumber
\end{align}
The statement of the lemma follows now by combining (\ref{ds1}),
(\ref{ds2}), and the basic identity
\begin{equation}
- \sum_{d|q} \frac{\mu(d)\log d}{d}  = \frac{\vphi(q)}{q} \sum_{p|q}
\frac{\log p}{p-1}.
\end{equation}
\end{proof}

It follows from the estimate (\ref{exsubq2}) that the integral
(\ref{exsubq1}) is infinite if and only if the series
\begin{equation}
\sum_{q\in\qq} \frac{\vphi(q) \log q}{q^2}\label{exsubq3}
\end{equation}
diverges.  Thus we state our conjecture in the following form.
\begin{conjecture}\label{exsubq4}
Let $\qq$ be a subset of positive integers.  Then the set $\D(\alpha)\cap\qq$
is infinite for almost all irrational points $\alpha$ in $\R/\Z$ if and
only if the series {\rm (\ref{exsubq3})} diverges.
\end{conjecture}

\noindent We note that the analogous statement for the principal
convergents to almost all $\alpha$ is a well known theorem of
Erd\"os \cite{erdos1970}.  Related questions of Diophantine approximation
are considered in \cite[Chapter 2]{harmon1998} and
\cite{vaaler1978}.


\section{Sequences of Martingale Differences}

Let $n\mapsto\beta_n$ be a bijective map from the set $\N$ of positive
integers onto the group $\Q/\Z$.  Then we say
that $\beta_1, \beta_2, \dots $ is an {\it enumeration} of the elements of
$\Q/\Z$.  For each positive integer
$n$ we define $\B_n$ to be the finite $\sigma$-algebra of subsets of
$\R/\Z$ generated by the components of the open set
\begin{equation}
\R/\Z\setminus\{\beta_1, \beta_2, \dots , \beta_n\}\label{sigmaalg1}
\end{equation}
together with the collection of singleton sets $\{\beta_1\}, \{\beta_2\},
\dots , \{\beta_n\}$.
Thus a function $g:\R/\Z\rar\C$ is $\B_n$-measurable if and only if it is
constant on each component of the open set
(\ref{sigmaalg1}).
Clearly we have $\B_1\subseteq\B_2\subseteq \cdots $.

In this section we determine a simple arithmetic condition that classifies
all orderings $\beta_1, \beta_2, \dots $ such that
$f_{\beta_n}$ is $\B_n$-measurable for each $n = 1, 2, \dots $ and the
sequence of functions and $\sigma$-algebras
\begin{equation}
\left\{(f_{\beta_n}, \B_n): n = 1, 2, \dots \right\}\label{marteq9}
\end{equation}
forms a sequence of martingale differences.  The martingale differences
which arise from this construction are a special case of
a general class of such functions considered by R.~F.~Gundy
\cite{gundy1966}.  These observations allow us to exploit results
from the theory of martingales to obtain metric theorems about the
continued fraction expansion of almost all real numbers.

Again let $\beta_1, \beta_2, \dots $ be an enumeration of the elements of
$\Q/\Z$.  We say that this enumeration
is {\it admissible} if it satisfies the following three conditions:
\begin{itemize}
\item[(i)] $\beta_1 = 0$,
\item[(ii)] if $m \ge 2$ and $\beta_k = \beta_m'$ then $k < m$,
\item[(iii)] if $m \ge 2$ and $\beta_l = \beta_m''$ then $l < m$.
\end{itemize}
It follows easily that an admissible enumeration of $\Q/\Z$ must begin
with either $0, \frac12, \frac13, \dots $,
or with $0, \frac12, \frac23, \dots $.

Suppose that $\beta_1, \beta_2, \dots $ is an enumeration
of $\Q/\Z$ such that $n\mapsto h(\beta_n)$ is nondecreasing.  Then
$\beta_1 = 0$, and using (\ref{hp1}) we find that
if $\beta_n$ is nonzero then $h(\beta_n') < h(\beta_n)$ and $h(\beta_n'')
< h(\beta_n)$.  It follows that such an enumeration
is admissible.  As an example, if we enumerate $\Q/\Z$ as
\begin{equation}
\tfrac01,\tfrac12,\tfrac13,\tfrac23,\tfrac14,\tfrac34,\tfrac15,\tfrac25,\tfrac35,\tfrac45,\tfrac16,\tfrac56,
	\tfrac17,\tfrac27,\tfrac37,\tfrac47,\tfrac57,\tfrac67, \dots ,\label{farey}
\end{equation}
then $n\mapsto h(\beta_n)$ is nondecreasing and the enumeration is
admissible.  Thus admissible enumerations of $\Q/\Z$
certainly exist.

The admissible enumeration (\ref{farey}) is constructed by arranging the
points of $\Q/\Z$ in order of increasing height, and
then ordering points of equal height by ordering their coset
representatives in $(0,1)$.  This construction
also leads to an admissible ordering for other naturally occurring
functions on $\Q/\Z$.  We describe such an admissible
enumeration associated to the Stern-Brocot tree (see \cite{graham1994}).
If $\beta$ is a rational number but not an integer
then $\beta$ has exactly two finite continued fraction expansions.  One
expansion has the form
\begin{equation}
\beta = [a_0; a_1, a_2, \dots , a_{N-1}, a_N],\quad\text{where}\ a_N \ge 2,\label{extra1}
\end{equation}
and then the other expansion is
\begin{equation*}
\beta = [a_0; a_1, a_2, \dots , a_{N-1}, a_N-1, 1].
\end{equation*}
We define
\begin{equation*}
s(\beta) = a_1 + a_2 + \cdots + a_N,
\end{equation*}
so that $s(\beta)$ is the sum of the partial quotients (other than $a_0$)
in both expansions of $\beta$.  If $n$ is an integer
we define $s(n) = 1$.  It is clear that $s$ is constant on cosets of
$\Q/\Z$ and thus it is well defined as a map
$s:\Q/\Z\rar \N$.  Clearly $s(\beta) = 1$ if and only if $\beta = 0$ in
$\Q/\Z$.  

Suppose that $\beta$ is rational number but not an integer, and $\beta$ has the finite continued fraction 
expansion (\ref{extra1}).  If $N$ is odd, then arguing as in the proof of Lemma \ref{farey0}, we get
\begin{equation*}
s(\beta') + a_N = s(\beta)\quad\text{and}\quad s(\beta'') + 1 = s(\beta).
\end{equation*}
If $N$ is even we find that
\begin{equation*}
s(\beta') + 1 = s(\beta)\quad\text{and}\quad s(\beta'') + a_N = s(\beta).
\end{equation*}
In particular, these identities show that if $\beta$ is a nonzero point in $\Q/\Z$, then $s(\beta') < s(\beta)$ and 
$s(\beta'') < s(\beta)$.  It follows that if $\beta_1, \beta_2, \dots $ is an enumeration of $\Q/\Z$
such that $n\mapsto s(\beta_n)$ is nondecreasing,
then the enumeration is admissible.  For example, the enumeration
\begin{equation}
\tfrac01,\tfrac12,\tfrac13,\tfrac23,\tfrac14,\tfrac25,\tfrac35,\tfrac34,\tfrac15,\tfrac27,\tfrac38,\tfrac37,\tfrac47,
	\tfrac58,\tfrac57,\tfrac45, \dots ,\label{sternbrocot}
\end{equation}
that corresponds to the ordering induced by the Stern-Brocot tree, is such
that $n\mapsto s(\beta_n)$
is nondecreasing.  Hence (\ref{sternbrocot}) is an admissible enumeration
of $\Q/\Z$, but distinct from (\ref{farey}).
In particular, the map $n\mapsto h(\beta_n)$ fails to be nondecreasing for
the enumeration (\ref{sternbrocot}).

\begin{theorem}\label{admordthm1}
Let $\beta_1, \beta_2, \dots $ be an enumeration of $\Q/\Z$.  Then
$\beta_1, \beta_2, \dots $ is admissible if and only if
for each positive integer $n$ the function $f_{\beta_n}$ is
$\B_n$-measurable.
\end{theorem}

\begin{proof}  First assume that $\beta_1, \beta_2, \dots $ is admissible.
 As $\beta_1 = 0$ it follows that
$f_{\beta_1} = f_0$ is constant.  Hence it is trivial that $f_{\beta_1}$
is $\B_1$-measurable.  Now suppose that
$n \ge 2$.  By hypothesis both $\beta_n'$ and $\beta_n''$ are elements of
the set $\{\beta_1, \beta_2, \dots , \beta_n\}$.
Therefore $\B_n$ contains the sub-$\sigma$-algebra $\A_n$ generated by the
components of the open set
$$\R/\Z\setminus\{\beta_n, \beta_n', \beta_n''\}$$
and the singleton sets $\{\beta_n\}, \{\beta_n'\}, \{\beta_n''\}$.  From the
definition (\ref{fbdef2}) it follows that $f_{\beta_n}$ is
$\A_n$-measurable and hence also $\B_n$-measurable.

Now assume that the function $f_{\beta_n}$ is $\B_n$-measurable for each
positive integer $n$.  In particular
the function $f_{\beta_1}$ must be constant on each component of the open
set $\R/\Z\setminus\{\beta_1\}$.  That is,
$f_{\beta_1}$ must be constant on $\R/\Z\setminus\{\beta_1\}$.  Hence
$f_{\beta_1}$ is constant almost
everywhere and therefore $\beta_1 = 0$.  Now assume that $n \ge 2$.  Then
$f_{\beta_n}$ is constant on
each of the open sets $I(\beta_n',\beta_n)$ and $I(\beta_n,\beta_n'')$.
Therefore the $\sigma$-algebra $\A_n$ defined
above is the smallest $\sigma$-algebra for which $f_{\beta_n}$ is
measurable.  It follows that $\A_n\subseteq \B_n$ and
therefore $\beta_n'$ and $\beta_n''$ must be elements of the set
$\{\beta_1, \beta_2, \dots , \beta_n\}$.  That is, if $\beta_k = \beta_n'$
then $k < n$, and
if $\beta_l = \beta_n''$ then $l < n$.  This shows that the enumeration
$\beta_1, \beta_2, \dots $ is admissible.
\end{proof}

\begin{lemma}\label{admordlem1}
Let $\beta_1, \beta_2, \dots $ be an admissible enumeration of $\Q/\Z$.
If $\beta_l$ and $\beta_m$ are distinct nonzero
points in $\Q/\Z$ such that $\beta_m$ is contained in
$I(\beta_l',\beta_l'')$, then $l < m$.
\end{lemma}

\begin{proof}  Let $h(\beta_l) = Q \ge 2$ and define
\begin{equation*}
M(\beta_l) = \{n\in\N:\beta_n\in I(\beta_l',\beta_l'')\}.
\end{equation*}
Then $M(\beta_l)$ is not empty and we may clearly assume that $m$ is the
smallest positive integer in $M(\beta_l)$.  Now
$I(\beta_l',\beta_l'')$ is a component interval in
$\R/\Z\setminus\F_{Q-1}$ and therefore $h(\beta_l) = Q \le h(\beta_m)$.
As
$$\beta_m \in I(\beta_l',\beta_l'')\cap I(\beta_m',\beta_m''),$$
it follows from Lemma \ref{farintlem} that either
\begin{equation*}
I(\beta_m',\beta_m'')\subseteq I(\beta_l',\beta_l)\quad\text{or}\quad
I(\beta_m',\beta_m'')\subseteq I(\beta_l,\beta_l'').
\end{equation*}
Assume that the first alternative
\begin{equation}
I(\beta_m',\beta_m'')\subseteq I(\beta_l',\beta_l)\label{IconI}
\end{equation}
holds.  Write $\beta_m' = \beta_j$
and $\beta_m'' = \beta_k$.  As the enumeration $\beta_1, \beta_2, \dots $
is admissible we have $j < m$ and $k < m$.
Therefore neither $\beta_m'$ nor $\beta_m''$ can belong to
$I(\beta_l',\beta_l)$.  From (\ref{IconI}) we conclude that
\begin{equation}
\beta_m' = \beta_l'\quad\text{and}\quad \beta_m'' = \beta_l.\label{IconI2}
\end{equation}
Since $\beta_1, \beta_2, \dots $ is admissible, the second identity in
(\ref{IconI2}) implies that $l < m$.
If the second alternative
\begin{equation*}
I(\beta_m',\beta_m'')\subseteq I(\beta_l,\beta_l'')
\end{equation*}
holds then the inequality $l < m$ follows in a similar manner.  This
proves the lemma.
\end{proof}

Next we recall that (\ref{marteq9}) is a sequence of martingale
differences if each function $f_{\beta_n}$ is
$\B_n$-measurable, and if for $n = 1, 2, \dots $ the conditional
expectation of $f_{\beta_{n+1}}$ with respect to
$\B_n$ is $0$ almost everywhere.  In the present setting, the conditional
expectation of $f_{\beta_{n+1}}$ with respect to
$\B_n$ is $0$ almost everywhere if and only if
\begin{equation}
\int_J f_{\beta_{n+1}}(x)\ \dx = 0 \label{condex}
\end{equation}
for each component $J$ of the open set $\R/\Z\setminus\{\beta_1, \beta_2,
\dots , \beta_n\}.$

\begin{theorem}\label{martdif1}  Let $\beta_1, \beta_2, \dots $ be an
admissible enumeration of $\Q/\Z$.  Then
the sequence of functions and $\sigma$-algebras
\begin{equation}
\left\{(f_{\beta_n}, \B_n): n = 1, 2, \dots \right\}\label{marteq10}
\end{equation}
is a sequence of martingale differences.
\end{theorem}

\begin{proof}  Let $n$ be a positive integer.  Then $\beta_{n+1}\not= 0$
and by hypothesis the points $\beta_{n+1}'$ and
$\beta_{n+1}''$ are contained in the set $\{\beta_1, \beta_2, \dots ,
\beta_n\}$.  Obviously $\beta_1 = 0$ is not
contained in the open set $I(\beta_{n+1}',\beta_{n+1}'')$.  If $2 \le m
\le n$
then by Lemma \ref{admordlem1} the point $\beta_m$ is not contained in
$I(\beta_{n+1}',\beta_{n+1}'')$.  It follows that
$I(\beta_{n+1}',\beta_{n+1}'')$ is a component of the open set
\begin{equation}
\R/\Z\setminus\{\beta_1, \beta_2, \dots , \beta_n\}.\label{sigmaalg2}
\end{equation}
Using the definition (\ref{fbdef2}) we find that
\begin{equation*}
\int_{I(\beta_{n+1}',\beta_{n+1}'')} f_{\beta_{n+1}}(x)\ \dx = 0.
\end{equation*}
As $f_{\beta_{n+1}}$ is supported on $\oI(\beta_{n+1}',\beta_{n+1}'')$, it
follows that (\ref{condex}) also holds whenever
$J \not= I(\beta_{n+1}',\beta_{n+1}'')$ is any other component of
(\ref{sigmaalg2}).  This proves the theorem.
\end{proof}

R.~F.~Gundy \cite{gundy1966} investigated a general class of martingale
differences called $H$-systems.  If
$\beta_1, \beta_2, \dots $ is an admissible enumeration of $\Q/\Z$, then
it follows from Theorem \ref{comporththm}
and \cite[Proposition 1.1]{gundy1966} that (\ref{marteq10}) is an example
of an $H$-system.  The following theorem
is \cite[Theorem 2.1(a)]{gundy1966} applied to the sequence
(\ref{marteq10}).

\begin{theorem}\label{cfserrep0}
Let $F:\R/\Z\rar \R\cup\{\pm\infty\}$ be a Borel measurable function that
is finite almost everywhere and let
$\beta_1, \beta_2, \dots $ be an admissible enumeration of $\Q/\Z$.  Then
there exist real numbers
$\{c(\beta_n): n = 1, 2, \dots \}$ such that
\begin{equation*}
\lim_{N\rar\infty} \sum_{n=1}^N c(\beta_n)f_{\beta_n}(x) =
F(x)\label{gundyeq2}
\end{equation*}
for almost all $x$ in $\R/\Z$.
\end{theorem}

In order to express results of this sort in terms of the continued
fraction expansion of $\alpha$, it is convenient
to select the admissible ordering of $\Q/\Z$ so that it has some
arithmetical significance.  Here we use the admissible
ordering (\ref{farey}).

\begin{proof}[Proof of Theorem \ref{cfserrep2}]  Let $\beta_1, \beta_2,
\dots $ be the admissible enumeration
(\ref{farey}).  By Theorem \ref{cfserrep0} there exist real numbers
$\{c(\beta):\beta\in\Q/\Z\}$ such that
\begin{equation}
\lim_{Q\rar\infty} \sum_{h(\beta) \le Q} c(\beta)\fb(x) =
F(x)\label{gundyeq3}
\end{equation}
for almost all $x$ in $\R/\Z$.  Assume that $\alpha$ is an irrational
point in $\R/\Z$ such that (\ref{gundyeq3}) holds.
For each positive integer $Q$ let $M = M(Q,\alpha)$ and $N = N(Q,\alpha)$
be the unique positive integers defined
by (\ref{cfraceq15}).  Using Theorem \ref{fbetacontfr} we can write
\begin{align}\label{sum4}
\begin{split}
\sum_{h(\beta) \le Q} c(\beta)\fb(\alpha)
	&= \sum_{n=1}^{N-1} (-1)^{n-1} q_{n-1}(\alpha) \sum_{\beta\in
E_n(\alpha)} c(\beta)\\
	&+ (-1)^{N-1} q_{N-1}(\alpha) \sum_{m=1}^M
		c\Big(\frac{mp_{N-1}(\alpha)+p_{N-2}(\alpha)}{mq_{N-1}(\alpha)+q_{N-2}(\alpha)}\Big).
\end{split}
\end{align}
We restrict the parameter $Q$ in (\ref{sum4}) to the subsequence of
denominators of the convergents to $\alpha$.  Along
this subsequence the positive integers $M$ and $N$ are related by the
identity $M = a_N$.  Thus (\ref{sum4}) reduces
to the simpler assertion (\ref{sum3}).
\end{proof}

Next we construct an example to show that the numbers
$\{c(\beta):\beta\in\Q/\Z\}$ that occur in the statement
of Theorem \ref{cfserrep2} are not uniquely determined by $F$.  We require
the following
lemmas, the first of which identifies the inverse of the maps
$\beta\mapsto\beta'$ and $\beta\mapsto\beta''$.

\begin{lemma}\label{primes}  Let $r'/s' < r/s < r''/s''$ be three
consecutive points in the set $\ff_s$ of Farey fractions
of order $s$.  Then we have
\begin{equation}
\{\beta\in\Q/\Z: \beta' = r/s\} = \Big\{\frac{mr+r''}{ms+s''}: m = 1, 2,
\dots \Big\}\label{prime1}
\end{equation}
and
\begin{equation}
\{\gamma\in\Q/\Z: \gamma'' = r/s\} = \Big\{\frac{nr+r'}{ns+s'}: n = 1, 2,
\dots \Big\}.\label{prime2}
\end{equation}
\end{lemma}

\begin{proof}  The identity (\ref{prime1}) follows immediately from the
inequalities
\begin{equation*}
\frac{r}{s} < \dots < \frac{mr+r''}{ms+s''} <
\frac{(m-1)r+r''}{(m-1)s+s''} < \dots < \frac{r+r''}{s+s''} <
\frac{r''}{s''},
\end{equation*}
and basic properties of Farey fractions.  Then (\ref{prime2}) is proved in
the same manner.
\end{proof}

\begin{lemma}\label{example}  Let $r'/s' < r/s < r''/s''$ be three
consecutive points in the set $\ff_s$
and write $\delta = r/s$ for the image of $r/s$ in $\Q/\Z$.  Define
\begin{equation*}
\beta_m = \frac{mr+r''}{ms+s''}\quad\text{for}\quad m = 0, 1, 2, \dots ,
\end{equation*}
and
\begin{equation*}
\gamma_n = \frac{nr+r'}{ns+s'}\quad\text{for}\quad n = 0, 1, 2, \dots .
\end{equation*}
Then $\delta' = r'/s'$ and $\delta'' = r''/s''$ in $\Q/\Z$, and for
positive integers $M$ and $N$ we have
\begin{equation}
\sum_{m=1}^M f_{\beta_m}(x) =  \begin{cases}M h(\delta) &\text{ if $x\in
I(\delta,\beta_M)$},\\
	- h(\delta'') &\text{ if $x\in I(\beta_M,\delta'')$},\\
	0 &\text{ if $x\notin\oI(\delta,\delta'')$}.\end{cases}\label{example2}
\end{equation}
and
\begin{equation}
\sum_{n=1}^N f_{\gamma_n}(x) =  \begin{cases}-N h(\delta) &\text{ if $x\in
I(\gamma_N,\delta)$},\\
	h(\delta') &\text{ if $x\in I(\delta',\gamma_N)$},\\
	0 &\text{ if $x\notin\oI(\delta',\delta)$}.\end{cases}\label{example3}
\end{equation}
\end{lemma}

\begin{proof}  That $\delta' = r'/s'$ and $\delta'' = r''/s''$ in $\Q/\Z$
follows the definition of
the maps $\beta\mapsto\beta'$ and $\beta\mapsto\beta''$, and the
hypothesis that $r'/s' < r/s < r''/s''$
are consecutive points in $\ff_s$.  From the definition (\ref{fbdef1}) we get
\begin{align}
\sum_{m=1}^M f_{\beta_m}(x) &= \sum_{m=1}^M h(\beta_m)\psi(x -
\beta_m)\label{example4}\\
	&\qquad\qquad -\sum_{m=1}^M h(\beta_m')\psi(x - \beta_m') - \sum_{m=1}^M
h(\beta_m'')\psi(x - \beta_m'').\nonumber
\end{align}
It follows using Lemma \ref{primes} that $\beta_m' = \delta$, $\beta_m'' =
\beta_{m-1}$ for each $m = 1, 2, \dots , M$,
and $\beta_0 = \delta''$.  These observations allow us to simplify
(\ref{example4}).  We find that
\begin{equation*}
\sum_{m=1}^M f_{\beta_m}(x) = h(\beta_M)\psi(x - \beta_M) - M
h(\delta)\psi(x - \delta) - h(\delta'')\psi(x - \delta''),
\end{equation*}
and this easily implies (\ref{example2}).  The identity (\ref{example3})
is established in a similar manner.
\end{proof}

Assume now that $\delta$ is a nonzero point in $\Q/\Z$.  Lemma
\ref{primes} and Lemma \ref{example} imply that
\begin{equation}
\lim_{Q\rar\infty}\Big\{f_{\delta}(x) - \sum_{\substack{h(\beta) \le
Q\\\delta = \beta'}} f_{\beta}(x)
	- \sum_{\substack{h(\gamma) \le Q\\\delta = \gamma''}}
f_{\gamma}(x)\Big\} = 0\label{example5}
\end{equation}
at almost all points $x$ in $\R/\Z$.  This shows that the numbers
$\{c(\beta):\beta\in\Q/\Z\}$ in the
statement of Theorem \ref{cfserrep2} are not uniquely determined by the
function $F$.


\section{Further Applications of the Martingale Property}

In this section we formalize the argument used to prove Theorem
\ref{cfserrep2} and derive further results about the
continued fraction expansion of almost all irrational points $\alpha$ in
$\R/\Z$.

Let $\beta_1, \beta_2, \dots $ be an admissible enumeration of $\Q/\Z$ and
let $\{c(\beta_n): n = 1, 2, \dots \}$ be a
collection of real numbers.  For $N = 1, 2, \dots $ we define
$S_N:\R/\Z\rar\R$ by
\begin{equation}
S_N(x) = \sum_{n=1}^N c(\beta_n) f_{\beta_n}(x).\label{sum-1}
\end{equation}
Theorem \ref{admordthm1} implies that each function $S_N$ is
$\B_N$-measurable, and
from Theorem \ref{martdif1} we conclude that the sequence of functions and
$\sigma$-algebras
\begin{equation}
\big\{(S_N,\B_N): N = 1, 2, \dots \big\}\label{martingale1}
\end{equation}
forms a martingale.  Of course this fact allows us to draw conclusions
about the behavior of the partial sums $S_N(x)$
for almost all $x$ as $N\rar\infty$.  As before we wish to express our
results in terms of the continued fraction
expansion of $\alpha$.  Therefore we use the admissible ordering
(\ref{farey}) and consider the subsequence of partial sums
\begin{equation}
T_Q(x) = \sum_{h(\beta) \le Q} c(\beta)\fb(x).\label{sum2}
\end{equation}
Clearly we have
\begin{equation*}
T_Q(x) = S_N(x), \quad\text{where}\quad N = N_Q = \sum_{q \le Q} \varphi(q).
\end{equation*}
It will be convenient to write $\M_Q = \B_{N_Q}$ for the corresponding
$\sigma$-algebra and $\M_0 = \{\emptyset, \R/\Z\}$ for
the trivial $\sigma$-algebra.  Thus a function
$g:\R/\Z\rar \C$ is $\M_Q$-measurable if it is measurable with respect to the
$\sigma$-algebra generated by the component intervals of
$\R/\Z\setminus\F_Q$ together with the singleton sets $\{\beta\}$
for $\beta$ in $\F_Q$.  Alternatively, $g$ is $\M_Q$-measurable if it is
constant on each component interval of
$\R/\Z\setminus\F_Q$.  It follows that the sequence of functions and
$\sigma$-algebras
\begin{equation}
\big\{(T_Q,\M_Q): Q = 1, 2, \dots \big\}\label{martingale2}
\end{equation}
forms a martingale.  It is instructive to 
note that if $g:\R/\Z\rar \R$ is an integrable function,
then the function (using a standard notation)
\begin{equation*}
g(x|\M_Q) = \int_{\R/\Z} K_Q(x,y) g(y)\ \dy
\end{equation*}
is the conditional expectation of $g$ given the $\sigma$-algebra $\M_Q$.
This is easily verified using Lemma \ref{KQeval}.

The following result is an application of Theorem \ref{fbetacontfr} and
the martingale convergence theorem.

\begin{theorem}\label{martthm1}  Let $\{c(\beta):\beta\in\Q/\Z\}$ be a
collection of real numbers and for each
positive integer $Q$ let $T_Q(x)$ be defined by {\rm (\ref{sum2})}.  If the
sequence of $L^1$-norms
\begin{equation}
\int_{\R/\Z} \bigl|T_Q(x)\bigr|\ \dx, \quad Q = 1, 2, \dots , \label{sum6}
\end{equation}
is bounded, then
\begin{equation}
\lim_{N\rar\infty} \sum_{n=1}^N (-1)^{n-1} q_{n-1}(\alpha) \sum_{\beta\in
E_n(\alpha)} c(\beta) = F(\alpha)\label{sum7}
\end{equation}
exists for almost all irrational points $\alpha$ in $\R/\Z$ and $\|F\|_1 <
\infty$.
\end{theorem}

\begin{proof}  By the martingale convergence theorem (see \cite[Section
4.2]{durrett2005} or \cite[Section 3.2]{garsia1970}),
the hypothesis (\ref{sum6}) implies that the limit
\begin{equation}
\lim_{Q\rar\infty} T_Q(x) = F(x)\label{limit1}
\end{equation}
exists for almost all $x$ in $\R/\Z$ and satisfies $\|F\|_1 < \infty$.
Suppose that $\alpha$ is an irrational point in
$\R/\Z$ for which (\ref{limit1}) holds.  For each positive integer $Q$ let
$M = M(Q,\alpha)$ and $N = N(Q,\alpha)$
be the unique positive integers defined by (\ref{cfraceq15}).  Using
Theorem \ref{fbetacontfr} we can write
\begin{align}
T_Q(\alpha) &= \sum_{n=1}^{N-1} (-1)^{n-1} q_{n-1}(\alpha) \sum_{\beta\in
E_n(\alpha)} c(\beta)\nonumber\\
	&\qquad\qquad	+ (-1)^{N-1} q_{N-1}(\alpha) \sum_{m=1}^M
		c\Big(\frac{mp_{N-1}(\alpha)+p_{N-2}(\alpha)}{mq_{N-1}(\alpha)+q_{N-2}(\alpha)}\Big).\label{sum8}
\end{align}
If $Q = q_N(\alpha)$ is a denominator of a convergent to $\alpha$ then
(\ref{sum8}) simplifies to
\begin{equation}
T_{q_N(\alpha)}(\alpha) = \sum_{n=1}^N (-1)^{n-1} q_{n-1}(\alpha)
\sum_{\beta\in E_n(\alpha)} c(\beta).\label{sum9}
\end{equation}
Now (\ref{sum7}) follows from (\ref{limit1}) and (\ref{sum9}).
\end{proof}

If $F:\R/\Z\rar\R$ is an integrable function and $T_Q$ is given by
\begin{equation}
T_Q(x) = \int_{\R/\Z} K_Q(x,y) F(y)\ \dy = \sum_{h(\beta)\le Q}
c(\beta)\fb(x),\label{ker1}
\end{equation}
where
\begin{equation*}
c(\beta) = \int_{\R/\Z} \fb(y) F(y)\ \dy,
\end{equation*}
then the conclusion (\ref{sum7}) follows from the much simpler Lemma
\ref{KQleb1}.  It is well known
(see \cite[Theorem 5.6, Section 4.5]{durrett2005}) that $T_Q$ has the form (\ref{ker1})
for some integrable function $F$
if and only if $T_Q$ converges to $F$ in $L^1$-norm as $Q\rar\infty$.  By
appealing to the martingale convergence
theorem we are able to establish (\ref{sum7}) under the weaker hypothesis
(\ref{sum6}).

Let $\delta$ be a nonzero point in $\Q/\Z$.  Using (\ref{example2}) and
(\ref{example3}) we find that
\begin{equation*}
\int_{\R/\Z} \Bigl|f_{\delta}(x) - \sum_{\substack{h(\beta) \le Q\\\delta
= \beta'}} f_{\beta}(x)
	- \sum_{\substack{h(\gamma) \le Q\\\delta = \gamma''}}
f_{\gamma}(x)\Bigr|\ \dx = 2h(\delta)^{-1},
\end{equation*}
and
\begin{equation*}
\int_{\R/\Z} \Bigl|\sum_{\substack{h(\beta) \le Q\\\delta = \beta'}}
f_{\beta}(x)
	+ \sum_{\substack{h(\gamma) \le Q\\\delta = \gamma''}}
f_{\gamma}(x)\Bigr|\ \dx \le 4h(\delta)^{-1},
\end{equation*}
for all positive integers $Q$.  Thus
\begin{equation*}
T_Q(x) = \sum_{\substack{h(\beta) \le Q\\\delta = \beta'}} f_{\beta}(x)
	+ \sum_{\substack{h(\gamma) \le Q\\\delta = \gamma''}} f_{\gamma}(x)
\end{equation*}
is an example that does not have the form (\ref{ker1}), but for which the
sequence (\ref{sum6}) of $L^1$-norms is  bounded.
Of course in this example we can establish the almost everywhere
convergence (\ref{example5}) directly without
appealing to Theorem \ref{martthm1}.

For the remainder of this section we consider the behavior of the partial
sums $T_Q(x)$ under the hypothesis that
$\beta\mapsto c(\beta)h(\beta)$ is bounded on $\Q/\Z$.  It will be
convenient to write
\begin{equation*}
T_Q(x) = \sum_{q=1}^Q U_q(x),\quad\text{where}\quad U_q(x) =
\sum_{h(\beta) = q} c(\beta)\fb(x).
\end{equation*}
If $\beta_1$ and $\beta_2$ are distinct points in $\Q/\Z$ with $h(\beta_1)
= h(\beta_2)$ then $\oI(\beta_1',\beta_1'')$
and $\oI(\beta_2',\beta_2'')$ intersect in a finite set, and therefore in
a set of measure zero.  From this
observation and the definition (\ref{fbdef2}) we find that
\begin{align}\label{ineq1}
\begin{split}
\tfrac12 q \max\{|c(\beta)|: h(\beta) = q\} &\le
\sup_{x\in\R/\Z}\bigl|U_q(x)\bigr| \\
	&= \|U_q\|_{\infty} \le q \max\{|c(\beta)|: h(\beta) = q\}.
\end{split}
\end{align}
This shows that $\beta\mapsto c(\beta)h(\beta)$ is bounded on $\Q/\Z$ if
and only if the sequence $\|U_q\|_{\infty}$
is bounded for $q = 1, 2, \dots $.

\begin{lemma}\label{stopping0}  For $q = 1, 2, \dots , Q$, let
$A_q\subseteq\R/\Z$ be an $\M_{q-1}$-measurable subset,
where $\M_0 = \{\emptyset, \R/\Z\}$ is the trivial $\sigma$-algebra.
Write $\chi_{A_q}$ for the characteristic function
of $A_q$.  Then we have
\begin{equation}
\int_{\R/\Z} \Big\{\sum_{q=1}^Q \chi_{A_q}(x) U_q(x)\Big\}^2\ \dx =
	\int_{\R/\Z} \Big\{\sum_{q=1}^Q \chi_{A_q}(x)U_q(x)^2\Big\}\
\dx.\label{intid1}
\end{equation}
\end{lemma}

\begin{proof}  We square out the integrand on the left of (\ref{intid1})
and integrate term by term.  Then it is clear that
the lemma will follow if we can verify that
\begin{equation}
\int_{A_q\cap A_r} U_q(x)U_r(x)\ \dx = 0 \label{intid2}
\end{equation}
whenever $1 \le q < r \le Q$.  Let $J$ be a component of
$\R/\Z\setminus\F_{r-1}$.  Then we have
\begin{equation*}
\int_J U_r(x)\ \dx = 0.
\end{equation*}
The function $U_q$ is
$\M_{r-1}$-measurable and therefore constant on $J$.
This implies that
\begin{equation}
\int_J U_q(x)U_r(x)\ \dx = 0. \label{intid3}
\end{equation}
By hypothesis the set $A_q\cap A_r$ is $\M_{r-1}$-measurable.  Therefore
it can be written as a finite disjoint union
of component intervals of $\R/\Z\setminus\F_{r-1}$ together with a set of
measure zero.  Thus (\ref{intid2}) follows
immediately from (\ref{intid3}), and the lemma is proved.
\end{proof}

\begin{theorem}\label{martthm2}  Let $\{c(\beta):\beta\in\Q/\Z\}$ be a
collection of real numbers such that
$\beta\mapsto c(\beta)h(\beta)$ is bounded on $\Q/\Z$.  Write $\cc$ for
the subset of points $\alpha$ in $\R/\Z$
such that
\begin{equation}
\lim_{Q\rar\infty} \sum_{q=1}^Q U_q(\alpha)\label{sum10}
\end{equation}
exists and is finite.  Write $\D$ for the subset of points $\alpha$ in
$\R/\Z$ such that both
\begin{equation}
\liminf_{Q\rar\infty} \sum_{q=1}^Q U_q(\alpha) = -\infty\quad\text{and}\quad
		\limsup_{Q\rar\infty} \sum_{q=1}^Q U_q(\alpha) = +\infty.\label{sum11}
\end{equation}
Write $\E$ for the subset of points $\alpha$ in $\R/\Z$ such that
\begin{equation}
\sum_{q=1}^{\infty} U_q(\alpha)^2 < \infty.\label{sum12}
\end{equation}
Then we have
\begin{equation}
{\rm (i)}\ |\cc\cup\D| = 1,\quad {\rm (ii)}\ |\cc\setminus\E| = 0,
	\quad\text{and}\quad {\rm (iii)}\ |\E\setminus\cc| = 0.\label{sets1}
\end{equation}
\end{theorem}

\begin{proof}  The sequence (\ref{martingale2}) forms a martingale, and by
hypothesis the increments
\begin{equation*}
\sup_{x\in\R/\Z} \bigl|T_q(x) - T_{q-1}(x)\bigr| = \|U_q\|_{\infty}
\end{equation*}
are bounded for $q = 1, 2, \dots $.  Therefore (i) follows from
\cite[Theorem 3.1, Section 4.3]{durrett2005}.

For positive integers $L$ and $Q$ let
\begin{equation*}
A(L,Q) = \{x\in\R/\Z:|T_q(x)| < L\ \text{for}\ q = 1, 2, \dots , Q-1\},
\end{equation*}
so that
\begin{equation*}
A(L,1) = \R/\Z \supseteq A(L,2) \supseteq A(L,3) \supseteq \cdots  .
\end{equation*}
It follows that $A(L,Q)$ is $\M_{Q-1}$-measurable and
\begin{equation*}
A(L,\infty) = \bigcap_{Q=1}^{\infty} A(L,Q) = \{x\in\R/\Z: |T_Q(x)| < L\
\text{for all}\ Q \ge 1\}.
\end{equation*}
Clearly we have
\begin{equation}
\cc \subseteq \bigcup_{L=1}^{\infty} A(L,\infty). \label{sets2}
\end{equation}
Next we define the stopping time $\eta_L:\R/\Z\rar\N\cup{\infty}$ by
\begin{equation*}
\eta_L(x) = \min\{Q: |T_Q(x)| \ge L\},
\end{equation*}
where $\eta_L(x) = \infty$ if $x$ belongs to $A(L,\infty)$.  Then we define
\begin{equation*}
T_Q^{(\eta_L(x))}(x) = \sum_{q=1}^{\min\{\eta_L(x), Q\}} U_q(x) =
\sum_{q=1}^Q \chi_{A(L,q)}(x) U_q(x),
\end{equation*}
so that for each positive integer $L$ the sequence
\begin{equation*}
\big\{(T_Q^{(\eta_L)},\M_Q): Q = 1, 2, \dots \big\}\label{martingale3}
\end{equation*}
forms a martingale, (see \cite[17.6 Corollary 2]{bauer1996}).  And by Lemma \ref{stopping0} we have
\begin{equation}
\int_{\R/\Z} \big\{T_Q^{(\eta_L(x))}(x)\big\}^2\ \dx =
	\int_{\R/\Z} \Big\{\sum_{q=1}^Q
\chi_{A(L,q)}(x)U_q(x)^2\Big\}\ \dx.\label{intid4}
\end{equation}
The inequality (\ref{ineq1}) and the assumption that $\beta\mapsto
c(\beta)h(\beta)$ is bounded on $\Q/\Z$
imply that
\begin{equation}
\sup\{\|U_q\|_{\infty}: q = 1, 2, \dots \} = M < \infty.\label{bound1}
\end{equation}
Thus we have
\begin{equation}
\bigl|T_Q^{(\eta_L(x))}(x)\bigr| \le L + M \label{martingale4}
\end{equation}
uniformly for all $x$ in $\R/\Z$ and all $Q = 1, 2, \dots $.  From the
martingale convergence theorem we conclude that
\begin{equation*}
\lim_{Q\rar\infty} T_Q^{(\eta_L(x))}(x) = V_L(x)
\end{equation*}
exists for almost all points $x$ in $\R/\Z$.  From (\ref{martingale4}) and
the dominated convergence theorem we obtain
\begin{equation*}
\lim_{Q\rar\infty} \int_{\R/\Z} \big\{T_Q^{(\eta_L(x))}(x)\big\}^2\ \dx =
\int_{\R/\Z} V_L(x)^2\ \dx < \infty.
\end{equation*}
Then from (\ref{intid4}) we find that
\begin{equation}
\lim_{Q\rar\infty} \int_{\R/\Z} \big\{T_Q^{(\eta_L(x))}(x)\big\}^2\ \dx
	=  \int_{\R/\Z} \Big\{\sum_{q=1}^{\infty}
\chi_{A(L,q)}(x)U_q(x)^2\Big\}\ \dx < \infty.\label{intid5}
\end{equation}
Now (\ref{intid5}) implies that
\begin{equation*}
\sum_{q=1}^{\infty} \chi_{A(L,q)}(x)U_q(x)^2 < \infty
\end{equation*}
for almost all $x$ in $\R/\Z$.  Hence we have
\begin{equation}
\sum_{q=1}^{\infty} U_q(x)^2 < \infty \label{intid6}
\end{equation}
for almost all $x$ in $A(L,\infty)$.  Then (\ref{sets2}) implies that
(\ref{intid6}) holds for almost all points $x$ in $\cc$.
This proves (ii).

The proof of (iii) is very similar.  For positive integers $L$ and $Q$ let
\begin{equation*}
B(L,Q) = \{x\in\R/\Z:\sum_{q=1}^{Q-1} U_q(x)^2 < L\},
\end{equation*}
so that
\begin{equation*}
B(L,1) = \R/\Z \supseteq B(L,2) \supseteq B(L,3) \supseteq \cdots  .
\end{equation*}
It follows that $B(L,Q)$ is $\M_{Q-1}$-measurable and
\begin{equation*}
B(L,\infty) = \bigcap_{Q=1}^{\infty} B(L,Q) = \{x\in\R/\Z:
\sum_{q=1}^{Q-1} U_q(x)^2 < L\ \text{for all}\ Q \ge 1\}.
\end{equation*}
And we have
\begin{equation}
\E \subseteq \bigcup_{L=1}^{\infty} B(L,\infty). \label{sets3}
\end{equation}
In this case we define a stopping time $\tau_L:\R/\Z\rar\N\cup{\infty}$ by
\begin{equation*}
\tau_L(x) = \min\{Q: \sum_{q=1}^Q U_q(x)^2 \ge L\},
\end{equation*}
where $\tau_L(x) = \infty$ if $x$ belongs to $B(L,\infty)$.  Then we write
\begin{equation*}
T_Q^{(\tau_L(x))}(x) = \sum_{q=1}^{\min\{\tau_L(x), Q\}} U_q(x) =
\sum_{q=1}^Q \chi_{B(L,q)}(x) U_q(x),
\end{equation*}
so that for each positive integer $L$ the sequence
\begin{equation*}
\big\{(T_Q^{(\tau_L)},\M_Q): Q = 1, 2, \dots \big\}\label{martingale5}
\end{equation*}
forms a martingale, (see \cite[17.6 Corollary 2]{bauer1996}).  By Lemma 
\ref{stopping0} we have
\begin{equation}
\int_{\R/\Z} \big\{T_Q^{(\tau_L(x))}(x)\big\}^2\ \dx =
	\int_{\R/\Z} \Big\{\sum_{q=1}^Q
\chi_{B(L,q)}(x)U_q(x)^2\Big\}\ \dx.\label{intid7}
\end{equation}
The bound (\ref{bound1}) implies that
\begin{equation*}
\sum_{q=1}^Q \chi_{B(L,q)}(x)U_q(x)^2 \le L + M^2
\label{martingale6}
\end{equation*}
uniformly for all $x$ in $\R/\Z$ and $Q = 1, 2, \dots $.  It follows from
(\ref{intid7}) and the martingale convergence
theorem that
\begin{equation*}
\lim_{Q\rar\infty} T_Q^{(\tau_L(x))}(x) = W_L(x)
\end{equation*}
exists and is finite for almost all points $x$ in $\R/\Z$.  Hence the limit
\begin{equation}
\lim_{Q\rar\infty} T_Q(x)\label{martingale7}
\end{equation}
exists and is finite for almost all points $x$ in $B(L,\infty)$.  Now
(\ref{sets3}) implies that the limit
(\ref{martingale7}) exists and is finite for almost all points $x$ in
$\E$.  This proves (iii).
\end{proof}

The conclusions (ii) and (iii) in the statement of Theorem \ref{martthm2}
are essentially the same as those obtained by
Gundy \cite[Theorem 3.1]{gundy1966} under somewhat different hypotheses.
In particular, Gundy works
with an $H$-system and a {\it regular} sequence of $\sigma$-algebras.  The
sequence of $\sigma$-algebras
$\M_Q$, $Q = 1, 2, \dots $, is not regular in Gundy's sense, but we are
able to establish the same type of result
by using instead the hypothesis that $\beta\mapsto c(\beta)h(\beta)$ is
bounded on $\Q/\Z$.

Finally, the series that occur in Theorem \ref{martthm2} can be rewritten
using the continued fraction interpretation of
the functions $f_{\beta}$.  In this way we arrive at the statement of
Theorem \ref{martthm0}.



\end{document}